\documentclass{amsart}
\usepackage{amssymb}
\usepackage{amsfonts}
\usepackage{amssymb}
\usepackage{amsmath}
\usepackage{amsthm}
\usepackage{enumerate}
\usepackage{tabularx}
\usepackage{centernot}
\usepackage{mathtools}
\usepackage{stmaryrd}
\usepackage{amsthm,amssymb}
\usepackage{etoolbox}
\usepackage{tikz}
\usepackage{amssymb}
\usetikzlibrary{matrix}
\usepackage{tikz-cd}
\usepackage{tikz}
\usepackage{marginnote}
\definecolor{mygray}{gray}{0.85}
\usepackage[linecolor=black, backgroundcolor=mygray,colorinlistoftodos,prependcaption,textsize=small]{todonotes}
\usepackage{xargs}

\renewcommand{\leq}{\leqslant}
\renewcommand{\geq}{\geqslant}

\renewcommand{\trianglelefteq}{\trianglelefteqslant}

\makeatletter
\def\subsection{\@startsection{subsection}{3}%
  \z@{.5\linespacing\@plus.7\linespacing}{.3\linespacing}%
  {\bfseries\centering}}
\makeatother

\makeatletter
\def\subsubsection{\@startsection{subsubsection}{3}%
  \z@{.5\linespacing\@plus.7\linespacing}{.3\linespacing}%
  {\centering}}
\makeatother

\makeatletter
\def\myfnt{\ifx\protect\@typeset@protect\expandafter\footnote\else\expandafter\@gobble\fi}
\makeatother

\newtheorem{theorem}{Theorem}

\newtheorem{corollary}[theorem]{Corollary}
\newtheorem{definition}[theorem]{Definition}
\newtheorem{lemma}[theorem]{Lemma}
\newtheorem{proposition}[theorem]{Proposition}

\newtheorem{fact}[theorem]{Fact}

\newcounter{claimcounter}
\newenvironment{claim}{\stepcounter{claimcounter}{\noindent {\bf Claim \theclaimcounter.}}}{}
\newenvironment{claimproof}[1]{\noindent{{\em Proof.}}\space#1}{\hfill $\rule{0.35em}{0.35em}$}

\begin{document}

\begin{abstract} We study the automorphism group of Hall's universal locally finite group $H$. We show that in $Aut(H)$ every subgroup of index $< 2^{\aleph_0}$ lies between the pointwise and the setwise stabilizer of a unique finite subgroup $A$ of $H$, and use this to prove that $Aut(H)$ is complete. We further show that $Inn(H)$ is the largest locally finite normal subgroup of $Aut(H)$. Finally, we observe that from the work of \cite{312} it follows that for every countable locally finite $G$ there exists $G \cong G' \leq H$ such that every $f \in Aut(G')$ extends to an $\hat{f} \in Aut(H)$ in such a way that $f \mapsto \hat{f}$ embeds $Aut(G')$ into $Aut(H)$. In particular, we solve the three open questions of Hickin on $Aut(H)$ from \cite{hickin}, and give a \mbox{partial answer to Question VI.5 of Kegel and Wehrfritz from \cite{kegel}.}
\end{abstract}

\title{The Automorphism Group of Hall's Universal Group}
\thanks{Partially supported by European Research Council grant 338821. No. 1106 on Shelah's publication list.}

\author{Gianluca Paolini}
\address{Einstein Institute of Mathematics,  The Hebrew University of Jerusalem, Israel}

\author{Saharon Shelah}
\address{Einstein Institute of Mathematics,  The Hebrew University of Jerusalem, Israel \and Department of Mathematics,  Rutgers University, U.S.A.}

\maketitle


\section{Introduction}

	In \cite{hall} Hall constructs a group $H$ with the following properties:
	\begin{enumerate}[(A)]
	\item $H$ is countable;
	\item $H$ is locally finite;
	\item $H$ embeds every finite group;
	\item any two isomorphic finite subgroups of $H$ are conjugate in $H$.
\end{enumerate}
	The group $H$ is unique modulo isomorphism and it is known as {\em Hall's universal locally finite group} (or simply as Hall's universal group). In model-theoretic terminology $H$ is an homogeneous structure, i.e. a structure $M$ such that every isomorphism between finitely generated substructures of $M$ extends to an automorphism of $M$. Groups of automorphisms of such structures have received extensive attention in the literature (see e.g. \cite{mac}, \cite{tent_mac}, \cite{lascar_hodges_shelah} and \cite{tent_ziegler}). Despite this, not much is known on $Aut(H)$. In this paper we make progress in this direction proving the following theorems:
	
	\begin{theorem}\label{SSIP} 
Every subgroup of $Aut(H)$ of index less than $2^{\aleph_0}$ lies between the pointwise  and the setwise stabilizer of a unique finite subgroup $A$ of $H$.
\end{theorem}

	\begin{theorem}\label{Aut(H)_complete} $Aut(H)$ is complete (i.e. $Aut(H)$ has no center and no outer automorphisms).
\end{theorem}

	\begin{theorem}\label{radical} $Inn(H)$ is the locally finite radical of $Aut(H)$ (i.e. it is the largest locally finite normal subgroup of $Aut(H)$).
\end{theorem}

	\begin{theorem}\label{univer_auto} For every countable locally finite $G$ there exists $G \cong G' \leq H$ such that every $f \in Aut(G')$ extends to an $\hat{f} \in Aut(H)$ in such a way that $f \mapsto \hat{f}$ embeds $Aut(G')$ into $Aut(H)$.
\end{theorem}


In particular, we solve the three open questions of Hickin on $Aut(H)$ from \cite{hickin} (see pg. 227), and give a partial answer to \mbox{Question VI.5 of Kegel and Wehrfritz from \cite{kegel}.}

	After the writing of this paper, thanks to the referee, we discovered that our Theorem \ref{Aut(H)_complete} is implied by a known result, i.e. that non-abelian simple groups have complete automorphism groups (see e.g. \cite{dyer}, where this is attributed to Burnside). In fact, by \cite[Theorem 6.1]{kegel}, $H$ is simple and so by the above we immediately get that $Aut(H)$ is complete. Nonetheless, we believe that our proof is enlightening and that the underlying ideas could be used to establish the completeness of the automorphism groups of other combinatorial and algebraic structures with the so-called strong small index property (cf. Definition \ref{SSIP}).

\section{The Strong Small Index Property for $Aut(H)$}

	In this section we prove Theorems \ref{SSIP} and \ref{univer_auto}.
	
\smallskip

	\begin{proof}[Proof of Theorem \ref{univer_auto}] This is implicitly proved in \cite[Claim 3.13(1) and 3.15]{312}.
\end{proof}

	As an immediate consequence of Theorem \ref{univer_auto}, we answer positively to the first two open questions of Hickin on $Aut(H)$ from \cite{hickin} (see pg. 227). 

	\begin{corollary} \begin{enumerate}[(1)]
	\item $Aut(H)$ embeds the symmetric group $Sym(\omega)$. 
	\item There is an infinite set $S \subseteq H$ such that every permutation of $S$ can be lifted to an automorphism of $H$.
\end{enumerate}
\end{corollary}

	\begin{proof} Let $G$ be the countably infinite dimensional vector space over the field of order $2$, and $G'$ and $F: f \mapsto \hat{f}$ as in Theorem \ref{univer_auto}. Let $S$ be a basis for $G'$ and $A(S)$ the subgroup of $Aut(G')$ of automorphisms induced by permutations of $S$. Then $F$ witnesses that every permutation of $S$ extends to an automorphism of $H$, and $F \restriction A(S)$ embeds $A(S) \cong Sym(\omega)$ into $Aut(H)$.
\end{proof}

%

\begin{definition}\label{SSIP} Let $M$ be a countable structure and $G = Aut(M)$. We say that $M$ (or $G$) has the {\em small index property} if every subgroup of $Aut(M)$ of index less than $2^{\aleph_0}$ contains the pointwise stabilizer of a finite set $A \subseteq M$.
\end{definition}

	\begin{proof}[Proof of Theorem \ref{SSIP}] We first show that $H$ has the small index property. By \cite[Theorem 6.9]{kechris} it suffices to show that $Aut(H)$ admits ample generics. To see this, by Sections 6.1 and 6.2 of \cite{kechris} it suffices to show that the class of finite groups has the extension property for partial automorphisms and the amalgamation property for automorphisms. The first follows directly from the corollary on pg. 538 of \cite{neumann}, and the second is proved in \cite[Claim 2.8]{312}. The theorem now follows from the small index property, the main result of \cite{ssip_canonical_hom} and \cite[Claim 2.8]{312}.
\end{proof}

\section{Completeness of $Aut(H)$}

	In this section we prove Theorem \ref{Aut(H)_complete}. To prove this we need the technology introduced in \cite{Sh_Pa_recon}, which we briefly review below.
	
\smallskip
	
	Let $H$ be Hall's group and $G = Aut(H)$. We denote by $\mathbf{A}(H) = \{ K \leq_{fin} H \}$ (where $K \leq_{fin} H$ means that $K \leq H$ and $K$ is finite), and by $\mathbf{EA}(H) = \{ (K, L) : K \in \mathbf{A}(H) \text{ and } L \leq Aut(K) \}$.  

	Let $(K, L) \in \mathbf{EA}(H)$, we define:
	$$ G_{(K, L)} = \{ h \in Aut(H) : h \restriction K \in L \}.$$
Notice that if $L = \{ id_K \}$, then  $G_{(K, L)} = G_{(K)}$, i.e. it equals the pointwise stabilizer of $K$, and that if $L = Aut(K)$, then $G_{(K, L)} = G_{\{ K \}}$, i.e. it equals the setwise stabilizer of $K$. We then let:
$$\mathcal{PS}(H) = \{ G_{(K)} : K \in \mathbf{A}(H) \} \; \text{ and } \; \mathcal{SS}(H) =\{ G_{(K, L)} : (K, L) \in \mathbf{EA}(H) \}.$$

Let $\mathbf{L}(H)$ be a set of finite groups such that for every $K \in \mathbf{A}(H)$ there is a unique $L \in \mathbf{L}(H)$ such that $L \cong Aut(K)$.

\begin{definition} We define the structure $ExAut(H)$, the {\em expanded group of automorphisms of $H$}, as follows:
	\begin{enumerate}[(1)]
	\item $ExAut(H)$ is a two-sorted structure;
	\item the first sort has set of elements $Aut(H) = G$;
	\item the second sort has set of elements $\mathbf{EA}(H)$;
	\item we identify $\{ (K, id_K) : K \in \mathbf{A}(H) \}$ with $\mathbf{A}(H)$;
	\item the relations are:
	\begin{enumerate}
	\item $P_{\mathbf{A}(H)} = \{ K \in \mathbf{A}(H)\}$ (recalling the above identification);
	\item for $L \in \mathbf{L}(H)$, $P_{L(H)} = \{ K \in \mathbf{A}(H) : Aut(K) \cong L \}$;
	\item $\leq_{\mathbf{EA}(H)} \; = \{ ((K_1, L_1), (K_2, L_2)): (K_i, L_i) \in \mathbf{EA}(H) \; (i =1,2)\text{, } K_1 \leq K_2 \text{ and } L_2 \restriction K_1 \leq L_1 \}$;
	\item $\leq_{\mathbf{A}(H)} \; = \{ (K_1, K_2): K_i \in \mathbf{A}(H) \; (i =1,2)\text{ and } K_1 \leq K_2 \}$;
	\item $P^{min}_{\mathbf{A}(H)} = \{ K \in \mathbf{A}(H): \{ e \} \neq K \in \mathbf{A}(H) \text{ is minimal in } (\mathbf{A}(H), \subseteq)\}$;
	\end{enumerate}
	\item the operations are:
	\begin{enumerate}[(f)]
	\item composition on $Aut(H)$;
	\end{enumerate}
	\begin{enumerate}[(g)]
	\item for $f \in Aut(H)$ and $K \in \mathbf{A}(H)$, $Op(f, K) = f(K)$;
	\end{enumerate}
	\begin{enumerate}[(h)]
	\item for $f \in Aut(H)$ and $(K_1, L_1) \in \mathbf{EA}(H)$, $Op(f, (K_1, L_1)) = (K_2, L_2)$ iff $f(K_1) = K_2$ and $L_2 = \{ f \restriction K_1 \pi f^{-1} \restriction K_2 : \pi \in L_1\}$.
	\end{enumerate}
	\end{enumerate}
\end{definition}
	
	We say that a set of subsets of a structure $N$ is second-order definable if it is preserved by automorphisms of $N$. We say that a structure $M$ is second-order definable in a structure $N$ if there is a injective map $\mathbf{j}$ mapping $\emptyset$-definable subsets of $M$ to second-order definable set of subsets $N$.

	\begin{theorem}\label{exp_auto_th}
	\begin{enumerate}[(1)]
	\item The map $\mathbf{j}_H = \mathbf{j} : (h, (K, L)) \mapsto (h, G_{(K, L)})$ witnesses second-order definability of $ExAut(H)$ in $Aut(H)$.
	\item Every $f \in Aut(G)$ has an extension $\hat{f} \in Aut(ExAut(H))$.
	\end{enumerate}
\end{theorem}

	\begin{proof} This is because of Theorem \ref{SSIP} and \cite[Theorem 12]{Sh_Pa_recon}.
\end{proof}

	Before proving Theorem \ref{Aut(H)_complete} we need a crucial lemma.

	\begin{lemma}\label{crucial_lemma} Let $K_1, K_2 \leq_{fin} H$ realizing the same quantifier-free type in $ExAut(H)$.
	\begin{enumerate}[(1)]
	\item If $K_1$ has prime order, then $K_1 \cong K_2$.
	\item If $K_1$ is abelian, then so is $K_2$.
	\item If $K_1$ is cyclic, then so is $K_2$.
	\item If $K_1$ is cyclic of order $n$, then $K_1 \cong K_2$.
	\item $K_1$ and $K_2$ have the same order.
	\item If $K_1$ and $K_2$ are with no center and $K_1$ is complete, then $K_1 \cong K_2$.
	\item If $K_1$ has no characteristic subgroup, then so does $K_2$.
	\item If $K_1$ is the alternating group on $n > 6$, then $K_1 \cong K_2$.
\end{enumerate}
\end{lemma}

	\begin{proof} (1) As groups of prime order are the only groups without non-trivial subgroups, and if $p \neq q$ are prime then $Aut(C_p) \not\cong Aut(C_q)$.
	
\smallskip

\noindent (2)  A finite group $K$ is abelian if and only if there is cyclic $L_1 \leq Aut(K)$ and $K^* \geq K$ such that for no $L_2 \leq Aut(K^*)$ we have $\{ f \restriction K : f \in L_2\} = L_1$.

\smallskip

\noindent (3)  A finite group $K$ is cyclic if and only it is abelian and there is a finite number of primes $P$ such that for every $p \in P$ there is a unique $K_1 \leq K$ of order $p$.

\smallskip

\noindent (4) By (4) it suffice to define $|K|$ for cyclic $K$. Let $|K| = \prod_{i < k} p_i^{n_i}$, for $(p_i)_{i < k}$ a sequence of primes with no repetitions and $n_i \geq 1$. Notice now the following:
\begin{enumerate}[(i)]
\item We can define $\{ p_i : i < k \}$;
\item For every $i < k$, we can define $\{ K' \leq K : p \mid |K'| \}$.
\item For every $i < k$, $|\{ K' \leq K : p_i \mid |K'| \}| = n_i$.
\end{enumerate}

\smallskip

\noindent (5) If $K_1$ is a finite group, then $|K_1| = 1 + \sum \{ m_{K} : K \leq K_1 \text{ cyclic} \}$, where, if $|K| = n$, $m_{K} = |\{a \in \{ 1, ..., n-1 \} : (a, n) = 1 \}$. Thus, by (4) we are done.

\smallskip

\noindent (6) By the choice of $ExAut(H)$ we have $Aut(K_1) \cong Aut(K_2)$. By (5) we have $|K_1| = |K_2|$. Hence, since $K_1$ is complete, $|Aut(K_2)| = |Aut(K_1)| = |K_1| = |K_2|$. Since $K_2$ is centerless we have $K_2 \cong Aut(K_2)$, and so we are done.

\smallskip

\noindent (7) By the choice of $ExAut(H)$ (cf. the operation $Op$).

\smallskip

\noindent (8) Since $K_1$ is the alternating group on $n > 6$, $K_1$ has no characteristic subgroup. Thus, by (7), also $K_2$ does not have a characteristic subgroup. Furthermore, by the proof of (2) and the fact that $K_1$ is not abelian, we have that $K_2$ is not abelian either. Hence, the center of $K_2$ is properly contained in $K_2$, and so it is the identity, since $K_2$ has no characteristic subgroup. Let $\pi_0: Alt(n) \cong K_1$, $\pi_1$ be an embedding of $Sym(n)$ into $H$ extending $\pi_0$ and $K_1^+ = ran(\pi_1)$. Let $K^+_2 \in \mathbf{A}(H)$ be such that $K_2 \leq K^+_2$, and $(K_1, K^+_1)$ and $(K_2, K^+_2)$ realize the same type. In particular, $|K^+_1| = |K^+_2|$ and  $[K^+_2 : K_2] = 2$. We claim that $K^+_2$ is centerless. In fact, suppose otherwise, and let $K^+_1 \leq K^{++}_1 \leq H$ be such that $K^{++}_1 = K^{+}_1 \oplus K_0$ with $|K_0| = 2$. Then $Aut(K^{++}_1) \cong Aut(K^{+}_1)$. But $K^+_2$ does not have such an extension, which contradicts the choice of $K^+_2$. Hence, $K^+_2$ is centerless, and so, by (6) and the fact that $n > 6$, there exists $\pi: K^+_1 \cong K^+_2$. Now, $K^+_1$ has a unique subgroup of index $2$, and so the same holds for $K^+_2$. Hence, $\pi$ has to map $K_1$ onto $K_2$.

\end{proof}

	We now prove Theorem \ref{Aut(H)_complete}.

	\begin{proof}[Proof of Theorem \ref{Aut(H)_complete}] Let $f \in Aut(G)$ and $\hat{f}$ the corresponding extension of $f$ to $Aut(ExAut(H))$. Now, $\hat{f}$ maps $P^{min}_{\mathbf{A}(H)} \cap P_{e(H)}$ onto itself, where $ e $ denotes the trivial group. Clearly,
$$P^{min}_{\mathbf{A}(H)} \cap P_{e(H)} = \{ K \leq H : |K| = 2 \},$$
since the groups of order $2$ are the only rigid groups without non-trivial subgroups.
Thus, $\hat{f}$ induces a permutation $g_1$ of $\mathcal{X}_2(H) = \{ x \in H : x \text{ has order } 2 \}$.
	
	\begin{claim} The map $g_1$ can be extended to a $g_2 \in G$.
\end{claim}

	\begin{claimproof} As $\mathcal{X}_2(H)$ generates $H$, it suffices to prove that if 
$x_1, ..., x_n \in \mathcal{X}_2(H)$ for $n > 3$, then there are $K_1, K_2 \leq_{fin} H$ such that:
\begin{enumerate}[(i)]
\item $x_1, ..., x_n \in K_1$;
\item $g_1(x_1), ..., g_1(x_{n}) \in K_2$;
\item there is an isomorphism $h$ from $K_1$ onto $K_2$ such that $\bigwedge_{0 < i \leq n} h(x_i) = g_1(x_i)$.
\end{enumerate}
Let $K_0$ be the subgroup of $H$ generated by  $\{ x_1, ..., x_n\}$ and $n_{*} = 2 |K_0|$. Then we can find $K_1 \geq K_0$ which is isomorphic to the alternating group on $n_*$. Thus, by Lemma \ref{crucial_lemma}, letting $K_2 = \hat{f}(K_1)$ we are done.
\end{claimproof}

\noindent Let $f_1 \in Aut(G)$ be such that $h \mapsto g_2 h g_2^{-1}$. We claim that $f_2 := f_1^{-1}f = id_G$. Towards contradiction, suppose there exists $h_1 \in G$ such that $h_2 := f_2(h_1) \neq h_1$. Since $\mathcal{X}_2(H)$ generates $H$, we can find $x_0 \in \mathcal{X}_2(H)$ such that:
$$x_1 : = h_1(x_0) \neq h_2(x_0) : = x_2.$$
Thus,
\[ \begin{array}{rcl}
	h_1G_{\{ e, x_0 \}}h_1^{-1} = G_{\{ e, x_1 \}} & 
	  \Rightarrow & f_2(h_1)f_2(G_{\{ e, x_0 \}})f_2(h_1^{-1}) = f_2(G_{\{ e, x_1 \}}) \\
	& \Rightarrow & h_2G_{\{ e, x_0 \}}h_2^{-1} = G_{\{ e, x_1 \}}\\
	& \Rightarrow & h_2(x_0) = x_1, 
\end{array}	\]
which is absurd. Hence, $f_2 = id_G$, and so $f = f_1 \in Inn(G)$, as wanted.
\end{proof}

\section{$Inn(H)$ is the Locally Finite Radical of $Aut(H)$}

	In this section we prove Theorem \ref{radical}, which solves the third question of Hickin\footnote{According to Hickin this question was posed by J. E. Roseblade, see \cite{hickin} pg. 227.} on $Aut(H)$ from \cite{hickin} (see pg. 227). We first need some facts and a proposition.

	\begin{fact}[\cite{centra}]\label{centralizers} Let $K \leq_{fin} H$. Then $\mathbf{C}_H(K)$ is isomorphic to an extension of $Z(K)$ by $H$ (i.e. $\mathbf{C}_H(K) / Z(K) \cong H$).
\end{fact}

	\begin{fact}[\cite{hickin_auto}{[Lemma 2.3]}]\label{hickin_fact} Let $A \lneq B$ and $C$ be finitely generated subgroups of an algebraically closed group $G$ and $f \in Aut(G) - Inn(G)$. Then there exists in $G$ an isomorphic copy $B'$ of $B$ over $A$ (i.e. $a' = a$ for every $a \in A$) such that $f(B') \not\subseteq \langle B', C \rangle_{G}$.
\end{fact}

	\begin{proposition}\label{crucial_prop} Let $f \in Aut(H) - Inn(H)$ be of finite order $n < \omega$, and $K \leq_{fin} H$. Then there are commuting $a \neq b \in \mathbf{C}_H(K) - K$ of order $2$ such that $f(a) = b$.
\end{proposition}

	\begin{proof} By Fact \ref{centralizers} we can find $a \in \mathbf{C}_H(K) - K$ of order $2$, since $H$ is generated by elements of order $2$. Similarly, letting $A = \langle f^{\pm i}(a), f^{\pm i}(K): i < n \rangle_H$, we can find $b'' \in \mathbf{C}_H(\langle f^{-1}(K), f^{-1}(a) \rangle_H) - A$ also of order $2$. Let now $A$ be as above, $B = \langle A, b'' \rangle_H$ and $C = \{ e \}$. Then, by Fact \ref{hickin_fact}, there exists $h: B' \cong_{A} B$ such that  $f(B') \not\subseteq B'$. Notice that $f(A) \subseteq A$, since $f$ is of finite order $n$. Thus, letting $b' = h(b'')$ and $f(b') = b$, we must have that $b \not\in B'$, and so in particular $b \neq a$ and $b \not\in K$, since $A \subseteq B'$. Furthermore, by the choice of $b''$ and that of $(A, B)$ we have:
\[ \begin{array}{rcl}
	b'' \in \mathbf{C}_H(\langle f^{-1}(K), f^{-1}(a) \rangle_H) 
            & \Rightarrow &  b' \in \mathbf{C}_H(\langle f^{-1}(K), f^{-1}(a) \rangle_H)\\
 			& \Rightarrow &  b \in \mathbf{C}_H(\langle K, a \rangle_H), \\
\end{array}	\]
since $B' \cong_{A} B$ and $\langle f^{-1}(K), f^{-1}(a) \rangle_H \leq A$.
\end{proof}

	Finally, we prove Theorem \ref{radical}. For $c \in H$ we denote conjugation by $c$ by $\square_{c}$.

	\begin{proof}[Proof of Theorem \ref{radical}] Let $N \trianglelefteq H$ properly containing $Inn(H)$, we want to show that $N$ is not locally finite. Let then $f \in N -Inn(H)$. If $f$ is of infinite order we are done. So suppose $f$ has finite order. We construct $g \in Aut(H)$ such that $g^{-1}f^{-1}gf$ has infinite order. 
%
Let $\{ d_i : i < \omega \} = H$. By induction on $i < \omega$, we define $K_i \leq_{fin} H$, $c_i \in H$, $(a_{2i-1}, a_{2i}), (b_{2i-1}, b_{2i}) \in H^2$ such that for $i < k < \omega$:
\begin{enumerate}[(i)]
	\item $f(a_i) = b_i$;
	\item $a_i \neq a_k$ and $b_i \neq b_k$ and $\{ a_j : j < i \} \cap \{ b_j : j < i \} = \emptyset$;
	\item $\langle a_j, b_j : j < i \rangle_H = \langle a_j: j < i \rangle_H \oplus \langle b_j : j < i \rangle_H \cong (\mathbb{Z}_2)^i \oplus (\mathbb{Z}_2)^i$;
	\item $K_i \leq K_{k}$;
	\item $(d_0, ..., d_{i-1}) \in K_i^{< \omega}$;
	\item $(a_0, ..., a_{2i-1}), (b_0, ..., b_{2i-1}), (c_0, ..., c_i) \subseteq K_i^{< \omega}$;
	\item $a_{2i}, b_{2i} \in \mathbf{C}_H(K_{i})$;
	\item $c_i \in \mathbf{C}_H(K_{i-1})$ and $\square_{c_i}$ maps $b_{2(i-1)}$ to $b_{2i-1}$ and $a_{2_i}$ to $a_{2i-1}$.
	\end{enumerate}
{\em Base Case}. Since $f \neq id_H$ and $H$ is generated by involutions, we can find $a_0 \neq b_0$ of order $2$ in $H$ such that $f(a_0) = b_0$. Let $c_0 = e$ and $K_0 = \{ e \}$.
\newline {\em Inductive Case}. Let $i > 0$, and suppose that $K_j \leq_{fin} H$, $c_j \in H$ and $(a_{2j-1}, a_{2j})$, $(b_{2j-1}, b_{2j}) \in H^2$ have been defined for every $j < i$. Using Proposition \ref{crucial_prop}, we find commuting $a_{2i-1} \neq b_{2i-1} \in \mathbf{C}_H(K_{i-1}) - K_{i-1}$ of order $2$ such that $f(a_{2i-1}) = b_{2i-1}$. Analogously, we find commuting $a_{2i} \neq b_{2i} \in \mathbf{C}_H(\langle K_{i-1}, a_{2i-1}, b_{2i-1} \rangle_H) - \langle K_{i-1}, a_{2i-1}, b_{2i-1} \rangle_H$ of order $2$ such that $f(a_{2i}) = b_{2i}$. Then, letting 
$$K^* = \langle K_{i-1}, a_{2(i-1)}, a_{2i-1}, a_{2_i}, b_{2(i-1)}, b_{2i-1}, b_{2_i} \rangle_H, $$
$$ K^* = K_{i-1} \oplus \langle a_{2(i-1)}, a_{2i-1}, a_{2_i}, b_{2(i-1)}, b_{2i-1}, b_{2_i} \rangle_H \cong K_{i-1} \oplus ((\mathbb{Z}_2)^3 \oplus (\mathbb{Z}_2)^3).$$ 
Let $\pi \in Aut(K^*)$ be such that $\pi$ is the identity on $K_{i-1}$ and it maps:
\begin{enumerate}[(1)]
	\item $a_{2(i-1)} \mapsto a_{2(i-1)}$ and $a_{2i-1} \mapsto a_{2_i} \mapsto a_{2i-1}$;
	\item $b_{2(i-1)} \mapsto b_{2i-1} \mapsto b_{2(i-1)}$ and $b_{2_i} \mapsto b_{2_i}$.
\end{enumerate}
Then there exists $c_i \in H$ such that $\square_{c_i}$ behaves as $\pi$ on $Aut(K^*)$. Finally, let $K_i = \langle K_{i-1}, a_{2(i-1)}, a_{2i-1}, b_{2(i-1)}, b_{2i-1}, d_{i-1}, c_i \rangle_H$. Then we fulfil the inductive requirements.

\smallskip

\noindent Let now, for $i < \omega$, $c^*_i = c_0 \cdots c_i$ and $g =  \lim(\square_{c^*_i}: i < \omega) \in Aut(H)$. Then for every $i < \omega$ we have:
$$a_{2i} \xmapsto{f} b_{2i} \xmapsto{g} b_{2i+1}\xmapsto{f^{-1}} a_{2i+1} \xmapsto{g^{-1}} a_{2i+2},$$
and so $g^{-1}f^{-1}gf$ has infinite order, as wanted.
\end{proof}

\end{document}